\documentclass{amsart}

\usepackage{color}

\newtheorem{theorem}{Theorem}[section]

\newtheorem{prop}[theorem]{Proposition}

\newtheorem{conj}[theorem]{Conjecture}

\theoremstyle{definition}
\newtheorem{definition}[theorem]{Definition}

\theoremstyle{remark}
\newtheorem{remark}[theorem]{Remark}

\numberwithin{equation}{section}

\def\bC{\mathbb{C}}
\def\cM{\mathcal{M}}
\def\cN{\mathcal{N}}

\def\cO{\mathcal{O}}

\def\bQ{\mathbb{Q}}

\newcommand{\Q}{\mathbb{Q}}

\def\bC{\mathbb{C}}
\def\cM{\mathcal{M}}

\def\cO{\mathcal{O}}

\def\bQ{\mathbb{Q}}

\begin{document}

\title{Hodge-theoretic variants of the Hopf and Singer Conjectures}

\author{Donu Arapura}
\address{Department of Mathematics, Purdue University, West Lafayette, IN 47907, USA.}
\email{arapura@purdue.edu}
\thanks{First author partially supported by a grant from the Simons Foundation}

\author{Laurentiu Maxim}
\address{Department of Mathematics, University of Wisconsin-Madison,  480 Lincoln Drive, Madison WI 53706-1388, USA.}
\email {maxim@math.wisc.edu}
\thanks{Second author 
acknowledges support from the project ``Singularities and Applications'' - CF 132/31.07.2023 funded by the European Union - NextGenerationEU - through Romania's National Recovery and Resilience Plan.}

\author{Botong Wang}
\address{Department of Mathematics, University of Wisconsin-Madison,  480 Lincoln Drive, Madison WI 53706-1388, USA.}
\email {wang@math.wisc.edu}

\subjclass[2020]{14F35, 14F45, 14C30, 14J45, 32S60}



\keywords{Singer-Hopf conjecture, aspherical manifold, Euler characteristic, sectional curvature, nef bundle, Fano variety}

\begin{abstract}
We propose several Hodge theoretic analogues of the conjectures of Hopf and Singer, and prove them in some special cases.
\end{abstract}

\maketitle

\section{Introduction}\label{intro}

\subsection{Singer-Hopf conjecture.} In 1931, Hopf made the following conjecture on the sign of the Euler characteristic of a Riemannian manifold, formulated here as strengthened by Singer in the context of aspherical manifolds, e.g., see \cite[Conjecture 25.1]{Gui}:
\begin{conj}[Singer-Hopf]\label{SH} If $X$ is a closed aspherical manifold of real dimension $2n$, then 
\begin{equation}\label{ho}(-1)^n  \cdot \chi(X) \geq 0.\end{equation} \end{conj}

Recall that a connected CW complex is said to be aspherical if its universal cover is contractible, so the above conjecture applies more generally to topological manifolds. Conjecture \ref{SH} is clearly true if $n=1$. 
Examples of aspherical manifolds include closed Riemannian manifolds with non-positive sectional curvature, for which the conjecture is true if $n=2$ since the Gauss-Bonnet integrand has the desired sign, cf. \cite[Theorem 5]{Ch}. However, Conjecture \ref{SH} is not yet known for all closed aspherical $4$-manifolds, and it is open for $n \geq 3$. 

Gromov \cite{Gro} introduced the notion of K\"ahler hyperbolicity, which includes compact K\"ahler manifolds with negative sectional curvature as a special case, and he verified Conjecture \ref{SH} for K\"ahler hyperbolic manifolds. Cao-Xavier \cite{CX} and  Jost-Zuo \cite{JZ} introduced the concept of K\"ahler nonellipticity, including compact K\"ahler manifolds with non-positive sectional curvature, and proved the Singer-Hopf conjecture in this case. See also \cite{Es}. These works proved the desired sign of the Euler characteristic by means of vanishing $L^2$-cohomology. More recently, Liu-Maxim-Wang \cite{LMW} proved the complex projective version of Conjecture \ref{SH} under the assumption that the (holomorphic) cotangent bundle $T^*X$ of $X$ is numerically effective (nef), e.g., globally generated. Moreover, it was conjectured in \cite[Section 6]{LMW} that aspherical complex projective manifolds have nef cotangent bundles. 
Building on ideas of \cite{DPS} and \cite{LMW}, Arapura-Wang \cite{AW} gave a new proof of Conjecture \ref{SH} for compact K\"ahler manifolds with non-positive sectional curvature, using the fact that such manifolds have nef cotangent bundles. Several generalizations of the Singer-Hopf conjecture in the singular context have been recently proposed in \cite[Conjectures 1.2 and 1.3]{Max} by using constructible functions.


\subsection{Hodge-theoretic variants of the Singer-Hopf conjecture.}
One of the aims of this note is to formulate Hodge-theoretic analogues of Conjecture \ref{SH}. For this purpose, we only consider compact K\"ahler (e.g., complex projective) manifolds. 

For a compact K\"ahler manifold $X$ of complex dimension $n$, the Euler characteristic can be computed as 
\begin{equation}\label{bb}\chi(X)=\sum_{p \geq 0} (-1)^p \cdot \chi^p(X),\end{equation}
with $\chi^p(X):=\chi(X, \Omega_X^p)$, and $\chi^p(X)= (-1)^n\chi^{n-p}(X)$ by Serre duality. 
Here, $ \Omega_X^p$ denotes the sheaf of holomorphic $p$-forms on $X$, and 
note that 
$\chi^0(X)$ is just the {Todd genus} of $X$.

We propose the following natural Hodge theoretic enhancement of the Singer-Hopf Conjecture \ref{SH}:
\begin{conj}\label{aw1}
If $X$ is a compact K\"ahler  
manifold of dimension $n$ which is aspherical or has a nef cotangent bundle, then for any $0\leq p \leq n$ one has:
\begin{equation}\label{aw2} (-1)^{n-p} \cdot \chi^p(X) \geq 0.\end{equation}
\end{conj}

Via the Riemann-Roch theorem, \eqref{aw2} yields inequalities for the Chern numbers of $X$ (of weighted degree $2n$), usually referred to as {\it Arakelov type inequalities}. Moreover, if $p=n$, \eqref{aw2} reduces to the  non-negativity of the holomorphic Euler characteristic $\chi(X, K_X)=(-1)^n \cdot \chi(X,\cO_X)$ of the canonical bundle $K_X$, 
which {Koll\'ar} \cite{Ko} conjectured for complex projective manifolds with generically large fundamental groups (we refer the reader to \cite{Ko} for the relevant definitions). Recall here that if $X$ is an aspherical projective manifold, then $\pi_1(X)$ is large,  e.g., see \cite[Proposition 6.7]{LMW}. Let us also mention that it was shown in \cite[Theorem 4]{ZQ} that if $X$ is projective with $\Omega^1_X$ nef and $\chi(X, K_X)>0$ then $K_X$ is ample, with the converse statement being conjectured in \cite{ZQ} and proved in the special case when the cotangent bundle is globally generated (cf. \cite[Theorem 3]{ZQ}) or if $\dim X \leq 4$.

Conjecture \ref{aw1} is implied by work of Gromov \cite{Gro} when $X$ is K\"ahler hyperbolic, and by Jost-Zuo \cite{JZ} in the K\"ahler nonelliptic case, see also \cite{Es}. The works \cite{Gro, Es, JZ} prove such inequalities via vanishing theorems for $L^2$-cohomology. 

More recently, Popa-Schnell \cite{PS} proved \eqref{aw2} for a complex projective manifold $X$ with a semi-small Albanese map, by showing a Nakano-type generic vanishing theorem for the sheaves $\Omega^p_X$ of holomorphic $p$-forms. Note that if the cotangent bundle is globally generated, then the Albanese map of $X$ is an immersion, and  \eqref{aw2}  follows in this case from work of Popa-Schnell \cite{PS}.

In Proposition \ref{caw1} we show that Conjecture \ref{aw1} holds for complex projective curves and surfaces with nef cotangent bundles, but see also Remark \ref{referee} for the case of aspherical smooth complex compact  surfaces. 
A similar argument can be used to show  the non-negativity of $\chi(X, K_X)$  for $X$ a complex projective manifold with nef cotangent bundle and complex dimension $n\leq 4$  (see Theorem \ref{34}).

\smallskip

One can further generalize the statement of Conjecture \ref{aw1} as follows:
\begin{conj}\label{aw1b}
If $X$ is a compact K\"ahler (or complex projective) manifold which is aspherical or has a nef cotangent bundle, and $\cM$ is a mixed Hodge module on $X$,  then for any integer $p$ one has:
\begin{equation}\label{aw2b} \chi(X, Gr_F^p DR(\cM)) \ge 0,\end{equation}
where $Gr_F^p DR(\cM)$ are the graded pieces, with respect to the Hodge filtration\footnote{Here we use the notation $Gr_F^p$ in place of $Gr^F_{-p}$ corresponding to the identification $F^p=F_{-p}$, which makes the transition between Saito's increasing filtration $F_{-p}$ appearing in the $D$-module language and the classical Hodge-theoretic situation of a decreasing filtration $F^p$.}, of the de Rham complex associated to  $\cM$.
\end{conj}

For instance, if $n=\dim X$ and $\cM=\bQ^H_X[n]$ is the constant Hodge module, then $Gr_F^p DR(\cM)\cong \Omega^p[n-p]$, so Conjecture \ref{aw1b} reduces in this case to Conjecture \ref{aw1}. Work of Popa-Schnell \cite{PS} on generic vanishing via mixed Hodge modules can be used to prove the inequality \eqref{aw2b} in the case when $X$ is a smooth subvariety of, or more generally, admits a finite morphism to, an abelian variety (see Theorem \ref{caw2}.). 


{\subsection{Hopf conjecture for positive sectional curvature and generalizations.}
We also propose a Hodge theoretic variant of another conjecture of Hopf, which can be stated as follows (see, e.g., \cite{Y2}).
\begin{conj}[Hopf]\label{ho2}
A compact, even-dimensional Riemannian manifold with positive sectional curvature has positive Euler characteristic. A compact, even-dimensional Riemannian manifold with non-negative sectional curvature has non-negative Euler characteristic.
\end{conj}

If $X$ is a compact K\"ahler  
manifold with non-negative (resp., positive) sectional curvature, then the holomorphic tangent bundle $TX$ is nef (resp., ample). So the following result from \cite{DPS}, which is derived from (semi-)positivity properties of nef/ample vector bundles (cf. \cite{DPS, FuLa, BG}), can  be seen as a generalization in the compact K\"ahler  context of Conjecture \ref{ho2}.

\begin{theorem}\label{hn}
If $X$ is a compact K\"ahler (resp., complex projective) manifold  
with a nef (resp., ample) tangent bundle $TX$ (e.g., $X$ has a non-negative (resp., positive) sectional curvature), then $\chi(X) \geq 0$ (resp., $>0$).
\end{theorem}

Let us remark that, in the case when $X$ has positive bisectional curvature (e.g., if $X$ has positive sectional curvature), Siu-Yau's proof \cite{SY} of Frankel's conjecture shows that $X$ is biholomorphic to a complex projective space. The same conclusion was proved by Mori \cite{Mor} for complex projective manifolds with ample tangent bundles. 
So the positive version of Theorem \ref{hn} can also be deduced from these deep results. Furthermore, a classification of all compact K\"ahler manifolds with  non-negative  bisectional curvature was obtained by Mok \cite{Mo} (see also \cite{Gu}), and this can be used as in Theorem \ref{19nn} to prove Theorem \ref{hn} in this situation. The general nef case follows, e.g., from \cite[Proposition 2.1]{DPS} and the Gauss-Bonnet theorem.

\medskip

In this note, we propose the following conjectural statement, which provides a Hodge-theoretic analogue of  Theorem \ref{hn}, and which holds, e.g., in the case of compact K\"ahler  manifolds with non-negative bisectional curvature (e.g., non-negative sectional curvature), see Theorem \ref{19nn} for a proof in this case.

\begin{conj}\label{ch}
If $X$ is a compact K\"ahler  manifold of dimension $n$ with a nef tangent bundle, then for $0 \leq p \leq n$ one has
\begin{equation} (-1)^p \cdot \chi^p(X) \geq 0,\end{equation}
with $\chi^p(X):=\chi(X, \Omega_X^p)$.
\end{conj}

As we discuss in Subsection \ref{Hodgef}, by results of \cite{DPS} it suffices to prove Conjecture \ref{ch} in the case when $X$ is a Fano manifold (i.e., a projective algebraic manifold with $K_X^{-1}$ ample). 
In Proposition \ref{Fano}, we prove this remaining case of  Conjecture \ref{ch} under the additional assumption that $X$ has an (algebraic) {\it cellular decomposition}, in the sense that there is a chain of Zariski closed subsets $X_i$ such that $X_i\setminus X_{i+1}$ is a union of affine spaces (see, e.g., \cite[Example 1.9.1]{Fu}). Examples of such varieties include {\it rational homogeneous} projective manifolds, i.e., projective quotients $G/P$, with $G$ a semi-simple Lie group and $P \subset G$ is a parabolic subgroup (see \cite[Section 3]{BT}). In fact, one has the following conjectural statement of Campana-Peternell \cite{CP}, which, in view of Proposition \ref{Fano}, would complete the proof  of Conjecture \ref{ch}.
\begin{conj}[Campana-Peternell] \label{cp} Any Fano manifold with nef tangent bundle is rational homogeneous.
\end{conj}
This is a natural generalization to the nef case of the above-mentioned result of Mori  \cite{Mor} (formerly known as the Hartshorne conjecture), and it is currently proved up to dimension $5$ and in several more cases;  see, e.g.,  \cite{DPP, Li, Wa0, Ka, Ka2, Wa, Mu2, Oc} for recent progress, as well as \cite{Mu} for a survey on this conjecture. Furthermore, in the unpublished preprint \cite{DPP}, Conjecture \ref{cp} is reduced to proving the non-vanishing of the top Segre class of $X$ (cf. \cite[Theorem 1.2, Conjecture 1.3]{DPP}).

Given that in order to prove Conjecture \ref{ch} we do not need the full strength of Conjecture \ref{cp}, we propose here the following weaker statements which imply Conjecture \ref{ch}.

\begin{conj}\label{Fan}
Any Fano manifold $X$ with nef tangent bundle has an (algebraic) cellular decomposition.
\end{conj}

\begin{conj}\label{Fan2}
If $X$ is a Fano manifold with nef tangent bundle, then its Hodge numbers are concentrated on the diagonal, i.e., $h^{p,q}(X)=0$ if $p \neq q$.  
\end{conj}

For example, Conjecture \ref{Fan} is true if $X$ has non-negative bisectional curvature. This is an application of Mok's proof of the generalized Frankel conjecture \cite{Mo}, using the fact that a Fano manifold is simply-connected. This observation already leads to the proof of Conjecture \ref{ch} when $X$ has non-negative bisectional curvature, see Theorem \ref{19nn}. Moreover, the proof of Proposition \ref{Fano} shows that Conjecture \ref{Fan} implies Conjecture \ref{Fan2}.

Let us also note that, by results of Carrell-Liebermann \cite{CaLi}, in order to prove Conjecture \ref{ch}, it would also suffice to show that $X$ admits a holomorphic vector field with only isolated zeros, see Remark \ref{CL}. 
This latter fact was conjectured by Carrell \cite{Ca} to be equivalent to $X$ being rational (see \cite{Kon, Hw} for progress on this conjecture).

We conclude this introduction with yet another conjecture, which may be of independent interest, and which is easily implied by either Conjecture \ref{Fan} or Conjecture \ref{Fan2}, so in particular it is true whenever any of these conjectures hold (e.g., if $X$ has non-negative bisectional curvature).

\begin{conj}\label{Fan3}
If $X$ is a Fano manifold with nef tangent bundle, then its odd Betti numbers vanish.
\end{conj}


\section{Proofs of results}

In this section, we discuss various aspects of the Hodge theoretic variants of the Singer and Hopf conjectures, namely Conjectures \ref{aw1}, \ref{aw1b} and \ref{ch} mentioned in the introduction. 

We begin by recalling the following definition. 
\begin{definition}
If $E$ is a vector bundle on a  projective (or compact complex) manifold $X$, denote by $\mathbf{P}(E)$ the projective bundle of hyperplanes in the fibers of $E$. 
A vector bundle $E$ on $X$ is called {\it ample} (resp. {\it nef}) if the line bundle $\mathcal{O}(1)$ on $\mathbf{P}(E)$ is ample (resp. nef). 
\end{definition}

In the complex manifold case, the nefness of a line bundle is considered in the sense of \cite[Definition 1.2]{DPS}, and this coincides with the usual definition when $X$ is projective. Note that globally generated bundles are nef. Properties of nef bundles are studied, e.g.,  in \cite{DPS, La}.

\subsection{On Conjecture \ref{aw1}}
In this section we restrict our attention to the projective context (but see also Remark \ref{referee}). We will make use of positivity properties for nef vector bundles, which we now recall.

\begin{definition}
The {\it Schur polynomial} $P_{\underline{a}} \in \mathbb{Z}[c_1,\ldots, c_n]$
of weighted degree $2n$ (with $\deg c_i=2i$) corresponding to a partition $\underline{a}=(a_1,\ldots,a_n)$ of $n$,  i.e., 
\begin{center} $n \geq a_1 \geq a_2 \geq \cdots \geq a_n \geq 0$ , with $\sum_j a_j=n$, \end{center}
is defined as the $n \times n$ determinant 
$$P_{\underline{a}}(c)=\det \left(c_{a_i-i+j} \right)_{1\leq i,j\leq n}.$$
Here, $c_i=0$ if $i \notin [0,n]$ and $c_0=1$.\footnote{The Schur polynomial corresponding to the partition $\underline{1}^n=(1,\ldots,1)$ is the Segre class $s_n$.}
\end{definition}

Then, as a special case of \cite[Theorem 2.5]{DPS}, one has the following.
\begin{theorem}\label{DPSpos} Let $X$ be a compact K\"ahler manifold of complex dimension $n$. Then the Schur polynomials $P_{\underline{a}}(c(E))$ of weighted degree $2n$ in the Chern classes of a nef vector bundle $E$ of rank $n$ on $X$ are non-negative, that is,
\begin{equation}\label{spos}\int_X P_{\underline{a}}(c(E)) \geq 0.\end{equation}
\end{theorem}

For simplicity, in what follows, we will drop the integral symbol.

The Riemann-Roch theorem, positivity results from \cite{DPS} as recalled in Theorem \ref{DPSpos}, and the Bogomolov-Miyaoka-Yau inequality\footnote{Almost all results concerning the Bogomolov-Miyaoka-Yau inequality are proved under the projective assumption. For a recent extension to compact K\"ahler manifolds under a suitable positivity condition on the canonical bundle, see \cite{Nom}.} \cite{Mi,Y,Ts} yield the following:
\begin{prop}\label{caw1}
Conjecture \ref{aw1} is true for smooth projective curves and surfaces with nef cotangent bundles. 
\end{prop}

\begin{proof}

In the calculations below, we use the notation $$c_i:=c_i(\Omega^1_X)=(-1)^ic_i(TX).$$ 

Assume $\dim X=1$. The Riemann-Roch theorem and the non-negativity result of Theorem \ref{DPSpos} yield:
$$\chi(X,\Omega^1_X)=-\chi(X,\cO_X)=\frac{c_1}{2} \geq 0.$$

Assume $\dim X=2$. 
The non-negativity of Schur polynomials of $T^*X$ implies that $c_2\geq 0$ and $c_1^2-c_2 \geq 0$. Together with the Riemann-Roch theorem, this yields:
\begin{equation}\label{rr1} \chi(X,\Omega^2_X)=\chi(X,\cO_X)=\frac{c^2_1+c_2}{12} \geq 0.\end{equation}
Also, one gets by Riemann-Roch that
\begin{equation}\label{rr2} \chi(X, \Omega^1_X) = \frac{c_1^2-5c_2}{6}. \end{equation}
The non-positivity of the above expression follows from $c_2\geq 0$, by using the Bogomolov-Miyaoka-Yau inequality (e.g., in the sense of \cite{Ts} since $K_X$ is nef,  hence $X$ is minimal), which in our notations can be written as: $$c_1^2\leq 3c_2.$$
\end{proof}

\begin{remark}\label{referee}  As pointed out by the referee (cf. also \cite[Remark 5.6]{ADCL}), Conjecture \ref{aw1} also holds for all aspherical smooth complex compact surfaces. Let us briefly indicate here the argument. Let $X$ be a smooth complex compact surface, and denote by $\chi$ and $\sigma$ its topological Euler characteristic and signature, respectively. Then using \eqref{rr1}, \eqref{rr2}, the fact that the Chern numbers $c_1^2$ and $c_2$ are the same for both $TX$ and $T^*X$, together with the standard identities $2\chi+3\sigma=c_1^2$ and $c_2=\chi$,  one gets that
\begin{equation}\label{rr3}
\chi(X,\Omega^2_X)=\chi(X,\cO_X)=\frac{1}{4}(\chi+\sigma), \ \ \ \ \chi(X, \Omega^1_X)=\frac{1}{2}(\sigma-\chi).
\end{equation}
Moreover, one has by \cite[Theorem 2]{JK} that $\chi \geq \vert \sigma \vert$, unless $X$ is a ruled surface over a curve of genus $\geq 2$. Since ruled surfaces are not aspherical, one then concludes that Conjecture \ref{aw1}  holds for any aspherical smooth complex compact surface $X$. Finally, 
as in \cite[Remark 5.6]{ADCL}, one further observes that if $X$ is an aspherical smooth compact complex surface of general type then the inequalities in Conjecture \ref{aw1} become strict.
\end{remark}

We next show that the arguments of Proposition \ref{caw1} can be extended to dimensions $n=3$ and $n=4$, to show the non-negativity of $\chi(X, K_X)$ in these dimensions. 

\begin{theorem}\label{34}
Let $X$ be a complex projective manifold of dimension $n\leq 4$ with a nef cotangent bundle. Then 
\begin{equation}\label{aw23} \chi(X, K_X)=(-1)^n \cdot \chi(X,\cO_X) \geq 0.\end{equation}
\end{theorem}

\begin{proof}
The case of curves and surfaces was already considered in 
Proposition \ref{caw1}.

Assume $\dim X=3$. The non-negativity of Schur polynomials of $\Omega^1_X$ implies:
$$P_{(3,0,0)}(c)=c_3\geq 0, \  P_{(2,1,0)}(c)=c_1c_2-c_3 \geq 0, \ P_{(1,1,1)}(c)=c_1^3-2c_1c_2+c_3 \geq 0.$$ Then, by using Riemann-Roch, we have
$$\chi(X,\Omega^3_X)=-\chi(X,\cO_X)=\frac{c_1c_2}{24}\geq 0.$$

Assume next that $\dim X=4$. The non-negativity of Schur polynomials of $\Omega^1_X$ implies:
$$P_{(4,0,0,0)}(c)=c_4\geq 0, \  P_{(3,1,0,0)}(c)=c_1c_3-c_4 \geq 0, \ P_{(2,2,0,0)}(c)=c_2^2-c_1c_3 \geq 0,$$
$$P_{(2,1,1,0)}(c)=c_1^2c_2-c_1c_3-c_2^2+c_4\geq 0,$$
$$P_{(1,1,1,1)}(c)=c_1^4-3c_1^2c_2+2c_1c_3+c_2^2-c_4\geq0.$$
A calculation via the Riemann-Roch theorem yields the following:
$$\chi(X,\Omega^4_X)=\chi(X,\cO_X)=\frac{-c_1^4+4c_1^2c_2+c_1c_3+3c_2^2-c_4}{720} .$$
Since $c_1^4$ has a negative coefficient in $\chi(X,\Omega^4_X)$, the non-negativity of the Schur polynomials is not  sufficient to guarantee that $\chi(X,\Omega^4_X) \geq 0$. However, since $\Omega^1_X$ is nef by our assumptions, it follows that $K_X$ is nef (hence $X$ is minimal), so the corresponding Miyaoka-Yau type inequality \cite{Ts,Zh} yields:
\begin{equation}\label{my} c_1^4\leq \frac{5}{2} c_1^2c_2.\end{equation}
We then have:
\begin{align*}
-c_1^4+4c_1^2c_2+c_1c_3+3c_2^2-c_4 & \geq  -c_1^4+4c_1^2c_2+c_1c_3+3c_2^2-c_4 - P_{(1,1,1,1)}(c) \\
& = -2c_1^4+7c_1^2c_2-c_1c_3+  2 c_2^2 \\
& \geq -2c_1^4+7c_1^2c_2-c_1c_3+ 2 c_2^2 -  2 P_{(2,2,0,0)}(c) \\
& = -2c_1^4+7c_1^2c_2   + c_1c_3\\
& \geq -2c_1^4+7c_1^2c_2  + c_1c_3 - P_{(3,1,0,0)}(c)\\
&  = -2c_1^4+7c_1^2c_2  + c_4 \\
&  \geq -2c_1^4+7c_1^2c_2  + c_4 - P_{(4,0,0,0)}(c)\\
& = -2c_1^4+7c_1^2c_2 \\
& \geq \frac{4}{5} c_1^4,
\end{align*}
where the last inequality follows by \eqref{my}. Finally, we have by \cite[Corollary 2.6]{DPS} that $c_1^4\geq 0$. Altogether, we get that $\chi(X,\Omega^4_X)=\chi(X,\cO_X) \geq 0$.
\end{proof}

\subsection{On Conjecture \ref{aw1b}}

In this section, we prove the following variant of Conjecture \ref{aw1b}. We refer to \cite{Sa0, Sa1} for background material on mixed Hodge modules.
\begin{theorem}\label{caw2}
Let $X$ be a 
complex projective manifold with a finite morphism $f\colon X \to A$ to
an abelian variety, and let $\cM$ be a mixed Hodge module on $X$. Then for any integer $p$ one has:
\begin{equation}\label{aw2bb} \chi(X, Gr_F^p DR(\cM)) \ge 0,\end{equation}
where $Gr_F^p DR(\cM)$ are the graded pieces, with respect to the Hodge filtration, of the de Rham complex associated to  $\cM$.
\end{theorem}
\begin{proof}
Since $f$ is a finite morphism, $\cN:=f_*\cM=f_!\cM$ is a mixed Hodge module on the abelian variety $A$. And since $f$ is proper and the functor $Gr_F^p DR$ commutes with proper pushforward, we have:
$$\chi(X, Gr_F^p DR(\cM))=\chi(A, Gr_F^p DR(\cN)).$$
Finally, by \cite[Corollary 2.7]{PS}, there exists $L \in Pic^0(A)$ such that  
\begin{center} $H^i(A, Gr_F^pDR(\cN)\otimes L)=0$, for all $i \neq 0$. \end{center}
Therefore,
$$\chi(A, Gr_F^p DR(\cN)) = \chi(A, Gr_F^p DR(\cN)\otimes L) \ge 0,$$
which proves the assertion.
\end{proof}

\subsection{On Conjecture \ref{ch}}\label{Hodgef}
In this section we discuss various aspects of Conjecture \ref{ch}. 

First, it was shown in \cite[Proposition 3.10]{DPS} that, if $TX$ is nef, then one of the two  situations occurs:

$(i)$ $c_1(X)^n=c_1(TX)^n=0$, 
 in which case it follows that all Chern polynomials of $X$ of weighted degree $2n$ vanish as well (cf. \cite[Corollary 2.7]{DPS}). In particular, by Riemann-Roch, $\chi(X,\Omega^p)=0$ for all $p$, and Conjecture \ref{ch} is true in this case.

$(ii)$ $c_1(X)^n>0$, in which case $X$ is a Fano manifold (i.e., a complex projective manifold with $K_X^{-1}$ ample). Then the Kodaira vanishing theorem implies that 
$$H^q(X,\cO_X)=H^q(X,K_X \otimes K_X^{-1})=0, \ \ \ {\rm for} \ q\geq 1.$$
Hence, $\chi^0(X)=\chi(X,\cO_X)=\dim H^0(X, \cO_X)=1$. By Serre duality, we further get that $(-1)^n \chi^n(X)>0$.

It therefore remains to prove Conjecture \ref{ch} in the Fano case for $0<p<n$.
In what follows, we prove this case of Conjecture \ref{ch} under the additional assumption that $X$ has a cellular decomposition, in the sense that there is a chain of Zariski closed subsets $X_i \subseteq X$ such that $X_i\setminus X_{i+1}$ is a union of affine spaces. More precisely, we have the following.

\begin{prop}\label{Fano}
If the complex projective manifold $X$ has a cellular decomposition (e.g., $X$ is rational homogenous),  then $(-1)^p \cdot \chi^p(X) >0$ for all  $p \le \dim X$.
\end{prop}

\begin{proof}
Since $X$ has a cellular decomposition,  the cycle map $cl_X:A^*(X) \to H^*(X)$ is an isomorphism, see \cite[Example 19.1.11(b)]{Fu}. In particular, the cohomology of $X$ is generated by algebraic cycles.  Therefore the Hodge numbers of $X$ are concentrated on the diagonal, i.e., $h^{p,q}=0$ unless $p=q$. Furthermore, it follows by Hard Lefschetz and the fact that $h^{0,0}=1$ that $h^{p,p} >0$ for all $p \le \dim X$. (An alternative proof of this fact follows along the lines of \cite[Theorem 3.1]{BE}, using induction over the number of cells together with the fact that $H^i_c(\bC^k;\bQ)=0$ for $i \neq 2k$ and $H^{2k}_c(\bC^k;\Q)=\bC$ is pure of weight $2k$.)
Therefore, $$(-1)^p \cdot \chi(X, \Omega^p) = h^{p,p} >0.$$
\end{proof}

Altogether, we obtain the following.
\begin{theorem}\label{19nn}
If $X$ is a compact K\"ahler 
manifold of dimension $n$ with non-negative bisectional curvature (e.g., non-negative sectional curvature), then for $0 \leq p \leq n$ one has
\begin{equation} (-1)^p \cdot \chi^p(X) \geq 0,\end{equation}
with $\chi^p(X):=\chi(X, \Omega_X^p)$.
\end{theorem}

\begin{proof} First recall that the bisectional curvature can be written as a positive linear combination of two sectional curvatures (e.g., see \cite{Z}). So if $X$ has non-negative sectional curvature, then it also has a non-negative bisectional curvature.

Secondly, our assumptions imply that the tangent bundle $TX$ is nef. Indeed,
a compact K\"ahler manifold $X$ with non-negative bisectional curvature has (by definition) a tangent bundle $TX$ which is Griffiths semipositive. It is also known that Griffiths semipositive bundles are nef, cf. \cite{DPS}.

As above, it suffices to further restrict our attention to the case when $X$ is a Fano manifold. Then $X$ is simply-connected, and hence, by Mok's theorem \cite[Main Theorem]{Mo},  $X$ is a product of complex projective spaces and compact Hermitian symmetric spaces (of rank $\geq 2$), all of which are known to admit a cellular decomposition. Hence $X$ admits a cellular decomposition, and the claim follows now by Proposition \ref{Fano}.
\end{proof}

\begin{remark}\label{CL} By results of Carrell-Liebermann \cite[Theorem 1]{CaLi}, an alternative approach to proving the diagonal concentration of Hodge numbers, and thus to prove Conjectures \ref{Fan2} and  \ref{ch}, is to show that $X$ admits a holomorphic vector field with only isolated zeros. Note that, if $X$ is as in $(ii)$ above, then Bott's residue formula implies that any vector field on $X$ has a non-empty vanishing locus. Moreover, if $TX$ is globally generated, Bertini's theorem tells us that the zero locus of a general vector field on $X$ consists of only isolated points.
\end{remark} 

\begin{remark}
As already mentioned in the introduction, Proposition \ref{Fano} can be seen as reducing the proof of Conjecture \ref{ch} to the Campana-Peternell Conjecture \ref{cp}. This also applies to the situation described in the previous remark, since rational homogeneous spaces are known to carry holomorphic vector fields with only isolated zeros. Moreover, while the nefness of the tangent bundle implies by \eqref{spos} the non-negativity of the top Segre class $s_n(X):=s_n(TX)$, it is known that rational homogeneous spaces have positive top Segre class \cite[Theorem 4.1]{DPP}. This motivates the reduction in \cite{DPP} of Conjecture \ref{cp} to the non-vanishing of the top Segre class.
\end{remark}

\section{Acknowledgments}

We thank the anonymous referees for carefully reading the manuscript and for making several constructive suggestions.


\bibliographystyle{amsalpha}

\end{document}